\documentclass{article}%
\usepackage{amsmath}
\usepackage{amsfonts}
\usepackage{amssymb}
\usepackage{graphicx}%
\providecommand{\U}[1]{\protect\rule{.1in}{.1in}}
\newtheorem{theorem}{Theorem}[section]

\newtheorem{conjecture}[theorem]{Conjecture}
\newtheorem{corollary}[theorem]{Corollary}

\newtheorem{lemma}[theorem]{Lemma}

\newtheorem{problem}[theorem]{Problem}
\newtheorem{proposition}[theorem]{Proposition}

\newenvironment{proof}[1][Proof]{\noindent\textbf{#1.} }{\ \rule{0.5em}{0.5em}}
\begin{document}

\author{Vadim E. Levit\\Department of Computer Science\\Ariel University, Israel\\levitv@ariel.ac.il
\and Eugen Mandrescu\\Department of Computer Science\\Holon Institute of Technology, Israel\\eugen\_m@hit.ac.il}
\title{$1$-well-covered graphs revisited}
\date{}
\maketitle

\begin{abstract}
A graph is \textit{well-covered} if all its maximal independent sets are of
the same size (M. D. Plummer, 1970). A well-covered graph is $1$%
-\textit{well-covered} if the deletion of every vertex leaves a graph which is
well-covered as well (J. W. Staples, 1975).

A graph $G$ belongs to class $W_{n}$ if every $n$ pairwise disjoint
independent sets in $G$ are included in $n$ pairwise disjoint maximum
independent sets (J. W. Staples, 1975). Clearly, $W_{1}$ is the family of all
well-covered graphs. It turns out that $G\in\mathbf{W}_{\mathbf{2}}$ if and
only if it is a $1$-well-covered graph without isolated vertices.

We show that deleting a \textit{shedding vertex} does not change the maximum
size of a maximal independent set including a given $A\in Ind\left(  G\right)
$ in a graph $G$, where $\mathrm{Ind}(G)$ is the family of all the independent
sets. Specifically, for well-covered graphs it means that the vertex $v$ is
shedding if and only if $G-v$ is well-covered. In addition, we provide new
characterizations of $1$-well-covered graphs, which we further use in building
$1$-well-covered graphs by corona, join, and concatenation operations.


\end{abstract}

\section{Introduction}

Throughout this paper $G=(V,E)$ is a simple (i.e., a finite, undirected,
loopless and without multiple edges) graph with vertex set $V=V(G)\neq
\emptyset$ and edge set $E=E(G)$. If $X\subset V$, then $G[X]$ is the graph of
$G$ induced by $X$. By $G-U$ we mean the subgraph $G[V-U]$, if $U\subset
V(G)$. We also denote by $G-F$ the subgraph of $G$ obtained by deleting the
edges of $F$, for $F\subset E(G)$, and we write shortly $G-e$, whenever $F$
$=\{e\}$.

The \textit{neighborhood} $N(v)$ of $v\in V\left(  G\right)  $ is the set
$\{w:w\in V\left(  G\right)  $ \textit{and} $vw\in E\left(  G\right)  \}$,
while the \textit{closed neighborhood} $N[v]$\ of $v$ is the set
$N(v)\cup\{v\}$. Let $\deg\left(  v\right)  =\left\vert N(v)\right\vert $ and
$\Delta\left(  G\right)  =\max\left\{  \deg\left(  v\right)  :v\in V\left(
G\right)  \right\}  $. If $\deg\left(  v\right)  =1$, then $v$ is a
\textit{leaf}. For an edge $ab\in E(G)$, let $G_{ab}=G\left[  V\left(
G\right)  -(N(a)\cup N(b))\right]  $.

The \textit{neighborhood} $N(A)$ of $A\subseteq V\left(  G\right)  $ is
$\{v\in V\left(  G\right)  :N(v)\cap A\neq\emptyset\}$, and $N[A]=N(A)\cup A$.
We may also use $N_{G}(v),N_{G}\left[  v\right]  ,N_{G}(A)$ and $N_{G}\left[
A\right]  $, when referring to neighborhoods in a graph $G$.

$C_{n},K_{n},P_{n},K_{p,q}$ denote respectively, the cycle on $n\geq3$
vertices, the complete graph on $n\geq1$ vertices, the path on $n\geq1$
vertices, and the complete bipartite graph on $p+q$ vertices, where $p,q\geq1$.

The \textit{disjoint union} of the graphs $G_{i},1\leq i\leq p$ is the graph
$G_{1}\cup G_{2}\cup\cdots\cup G_{p}$ having the disjoint unions $V(G_{1})\cup
V(G_{2})\cup\cdots\cup V\left(  G_{p}\right)  $ and $E(G_{1})\cup E(G_{2}%
)\cup\cdots\cup E\left(  G_{p}\right)  $ as a vertex set and an edge set,
respectively. In particular, $pG$ denotes the disjoint union of $p>1$ copies
of the graph $G$.

A \textit{matching} is a set $M$ of pairwise non-incident edges of $G$. If
$A,B\subset V(G)$ and every vertex of $A$ is matched by $M$ with some vertex
of $B$, then we say that $A$ is matched into $B$. A matching of maximum
cardinality, denoted $\mu(G)$, is a \textit{maximum matching}.

A set $S\subseteq V(G)$ is \textit{independent} if no two vertices from $S$
are adjacent, and by $\mathrm{Ind}(G)$ we mean the family of all the
independent sets of $G$. An independent set of maximum size is a
\textit{maximum independent set} of $G$, and $\alpha(G)=\max\{\left\vert
S\right\vert :S\in\mathrm{Ind}(G)\}$. Let $\Omega(G)$ denote the family of all
maximum independent sets.

\begin{theorem}
\cite{Berge1982}\label{th6} In a graph $G$, an independent set $S$ is maximum
if and only if every independent set disjoint from $S$ can be matched into $S$.
\end{theorem}

A graph $G$ is \textit{quasi-regularizable} if one can replace each edge of
$G$ with a non-negative integer number of parallel copies, so as to obtain a
regular multigraph of degree $\neq0$ \cite{Berge1982}. Equivalently, $G$ is
quasi-regularizable if and only if\textit{ }$\left\vert S\right\vert
\leq\left\vert N\left(  S\right)  \right\vert $ holds for every independent
set $S$ of $G$ \cite{Berge1982}. A graph $G$ is \textit{regularizable} if by
multiplying each edge by a positive integer, one gets a regular multigraph of
degree $\neq0$ \cite{Berge 1978}. For instance, every odd cycle $C_{2k+1}%
,k\geq2$, is regularizable.

\begin{theorem}
\cite{Berge 1978}\label{th11} \emph{(i)} Let $G$ be a connected graph that is
not a bipartite with partite sets of equal size. Then $G$ is regularizable if
and only if $\left\vert N(S)\right\vert >\left\vert S\right\vert $\ for every
non-empty independent set $S\subseteq V\left(  G\right)  $.

\emph{(ii)} A graph $G$ is regularizable if and only if $\left\vert
N(S)\right\vert \geq\left\vert S\right\vert $ for each independent set $S$,
and $\left\vert N(S)\right\vert =\left\vert S\right\vert \Rightarrow N\left(
N(S)\right)  =S$.
\end{theorem}

A graph is \textit{well-covered} if all its maximal independent sets are of
the same cardinality \cite{plum}. In other words, a graph is well-covered if
every independent set is included in a maximum independent set. It is known
that every well-covered graph is quasi-regularizable \cite{Berge1982}. If $G$
is well-covered, without isolated vertices, and $\left\vert V\left(  G\right)
\right\vert =2\alpha\left(  G\right)  $, then $G$ is a \textit{very
well-covered graph} \cite{Favaron1982}. The only well-covered cycles are
$C_{3}$, $C_{4}$, $C_{5}$ and $C_{7}$, while $C_{4}$ is the unique very
well-covered cycle.

A well-covered graph (with at least two vertices) is $1$-\textit{well-covered}
if the deletion of every vertex of the graph leaves a graph, which is
well-covered as well \cite{StaplesThesis}. For instance, $K_{2}$ is
$1$-well-covered, while $P_{4}$ is very well-covered, but not $1$-well-covered.

Let $n$ be a positive integer. A graph $G$ belongs to class $W_{n}$ if every
$n$ pairwise disjoint independent sets in $G$ are included in $n$ pairwise
disjoint maximum independent sets \cite{StaplesThesis}. First, if $V\left(
G\right)  =\emptyset$, then $G\in W_{n}$. Second, $K_{n}\in W_{n}$, for every
$n$. Third, $W_{1}\supseteq W_{2}\supseteq W_{3}\supseteq\cdots$, where
$W_{1}$ is the family of all well-covered graphs.\ A number of ways to build
graphs belonging to class $W_{n}$ are presented in \cite{StaplesThesis}.

\begin{theorem}
\cite{Staples}\label{th444} $G\in\mathbf{W}_{2}$ if and only if $\alpha
(G-v)=\alpha(G)$ and $G-v$ is well-covered, for every $v\in V(G)$.
\end{theorem}

A classification of triangle-free planar graphs in $\mathbf{W}_{\mathbf{2}}$
appears in \cite{Pinter1995}.

\begin{theorem}
\cite{Hoang2016a} Let $G$ be a triangle-free graph without isolated vertices.
Then $G$ is in $\mathbf{W}_{\mathbf{2}}$ if and only if $G_{ab}$ is
well-covered with $\alpha(G_{ab})=\alpha(G)-1$ for all edges $ab$.
\end{theorem}

A characterization of triangle-dominating\textit{ }graphs (i.e., graphs where
every triangle is also a dominating set) from $\mathbf{W}_{2}$ in terms of
forbidden configurations is presented in \cite{Hoang2016b}. $1$-well-covered
circulant graphs are considered in \cite{EVMVT2015}. Edge-stable equimatchable
graphs, which are, actually, $1$-well-covered line graphs, are investigated in
\cite{DE2016}.

By identifying the vertex $v_{i}$ with the variable $v_{i}$ in the polynomial
ring $R=K[v_{1},...,v_{n}]$ over a field $K$, one can associate with $G$ the
\textit{edge ideal} $I(G)=\left\{  v_{i}v_{j}:v_{i}v_{j}\in E(G)\right\}  $. A
graph $G$ is \textit{Cohen-Macaulay} (\textit{Gorenstein}) over $K$, if
$R/I(G)$ is a Cohen-Macaulay ring (a Gorenstein ring, respectively).

There are intriguing connections between graph theory and combinatorial
commutative algebra and graph theory \cite{Vi}. Consider, for instance, an
interplay between Cohen-Macaulay rings and graphs, were well-covered graphs
are known as unmixed graphs or may be reconstructed from pure simplicial
complexes \cite{Do,Woodroofe2011}. Even more fruitful interactions concern
shellability, vertex decomposability and well-coveredness
\cite{Biyi,CaCrRey2016,Estrada}. For example, every Cohen--Macaulay graph is
well-covered, while each Gorenstein graph without isolated vertices belongs to
$\mathbf{W}_{\mathbf{2}}$ \cite{Hoang2015}. Moreover, a triangle-free graph
$G$ is Gorenstein if and only if every non-trivial connected component of $G$
belongs to $\mathbf{W}_{\mathbf{2}}$ \cite{Hoang2016a}.

In this paper, we concentrate on structural properties of the class of
$1$-well-covered graphs, which is slightly larger than the class
$\mathbf{W}_{\mathbf{2}}$. Actually, we show that $G\in\mathbf{W}_{\mathbf{2}%
}$ if and only if it is a $1$-well-covered graph without isolated vertices. In
addition, a number of new characterizations for $1$-well-covered graphs is
provided. Some of them are based on the assumption that $G$\ is already
well-covered, while the others are formulated for general graphs. Further, we
describe shedding vertices as the ones that have no impact on the strength of
enlargement of independent sets. Specifically, for well-covered graphs it
means that the vertex $v$ is shedding if and only if $G-v$ is well-covered. We
also determine when corona, join, and concatenation of graphs are $1$-well-covered.

\section{Structural properties}

It is clear that $\alpha\left(  G-v\right)  \leq\alpha\left(  G\right)  $
holds for each $v\in V\left(  G\right)  $. If $u\in N\left(  v\right)  $ and
$G$ is well-covered, then there is some maximum independent set $S$ such that
$\left\{  u\right\}  \subset$ $S$. Hence $v\notin S$ and this implies
\[
\alpha\left(  G\right)  =\left\vert S\right\vert \leq\alpha\left(  G-v\right)
\leq\alpha\left(  G\right)  .
\]
In other words we get the following.

\begin{lemma}
\label{lem1}If $G$ is well-covered and $v\in V\left(  G\right)  $ is not
isolated, then $\alpha\left(  G-v\right)  =\alpha\left(  G\right)  $.
\end{lemma}

The converse of Lemma \ref{lem1} is not generally true. For instance,
$\alpha\left(  P_{6}-v\right)  =\alpha\left(  P_{6}\right)  $ holds for each
$v\in V\left(  P_{6}\right)  $, but $P_{6}$ is not well-covered.

According to Theorem \ref{th444}, no graph in class $\mathbf{W}_{2}$ may have
isolated vertices, since all these vertices are included in each of its
maximum independent sets. However, a graph having isolated vertices may be
$1$-well-covered; e.g., $K_{3}\cup K_{1}$. The following theorem shows, among
other things, that a graph is $1$-well-covered if and only if each of its
connected components different from $K_{1}$ is in class $\mathbf{W}%
_{\mathbf{2}}$.

\begin{theorem}
\label{th1}For every graph $G$ having no isolated vertices, the following
assertions are equivalent:

\emph{(i) }$G\neq P_{3}$ and $G-v$ is well-covered, for every $v\in V(G)$;

\emph{(ii)} $G$ is $1$-well-covered;

\emph{(iii)} $G$ is in the class $\mathbf{W}_{2}$;

\emph{(iv)} for each non-maximum independent set $A$ in $G$ there are at least
two disjoint independent sets $B_{1},B_{2}$ such that $A\cup B_{1},A\cup
B_{2}\in\Omega\left(  G\right)  $;

\emph{(v)} for every non-maximum independent set $A$ in $G$ there are at least
two different independent sets $B_{1},B_{2}$ such that $A\cup B_{1},A\cup
B_{2}\in\Omega\left(  G\right)  $;

\emph{(vi)} for each pair of disjoint non-maximum independent sets $A,B$ in
$G$, there exists some $S\in\Omega(G)$ such that $A\subset S$ and $B\cap
S=\emptyset$;

\emph{(vii)} for every non-maximum independent set $A$ in $G$ and $v\notin A$,
there exists some $S\in\Omega(G)$ such that $A\subset S,v\notin S$.
\end{theorem}

\begin{proof}
\emph{(i)} $\Rightarrow$ \emph{(ii) }Let $G\neq P_{3}$ be a graph such that
$G-v$ is well-covered, for every $v\in V(G)$.

In order to show that $G$ is $1$-well-covered, it is sufficient to show that
$G$ is well-covered. Suppose, to the contrary, that $G$ is not well-covered,
i.e., there is some maximal independent set $A$ in $G$ such that
$A\notin\Omega(G)$. Let $v\in V(G)-A$. Since $A$ is a maximal independent set
also in $G-v$, and $G-v$ is well-covered, it follows that $\alpha
(G-v)=\left\vert A\right\vert <\alpha(G)$. Hence, we get that $\alpha
(G-v)=\alpha(G)-1$, because, in general, $\alpha(G)-1\leq\alpha(G-v)$.
Consequently, every $v\in V(G)-A$ belongs to all maximum independent sets of
$G$. Therefore, $B=V(G)-A$ is an independent set in $G$, included in each
$S\in\Omega(G)$. It follows that $G$ is bipartite, with the bipartition
$\{A,B\}$. Since\textbf{ }$G$\textbf{ }is connected, $N(v)\cap B\neq\emptyset$
holds for every $v\in A$, and because, in addition, each maximum independent
set of $G$ contains $B$, it follows that $\Omega(G)=\{B\}$.

Let $a\in A$. Then $G-a$ is well-covered with
\[
\alpha(G-a)=\alpha(G)=\left\vert B\right\vert =\left\vert A\right\vert +1.
\]
Since $A-\{a\}$ is independent, it is possible to enlarge it to a maximum
independent set in $G-a$. Thus there exist $b_{1},b_{2}\in B$ such that
$\left(  A-\{a\}\right)  \cup\{b_{1},b_{2}\}$ is a maximum independent set in
$G-a$. Hence, $\left(  A-\{a\}\right)  \cup\{b_{1},b_{2}\}\in\Omega(G)$,
because $\left\vert \left(  A-\{a\}\right)  \cup\{b_{1},b_{2}\}\right\vert
=\alpha(G)$. Consequently, $\left(  A-\{a\}\right)  \cup\{b_{1},b_{2}\}=B$.
Finally, $A=\{a\}$ and $B=\{b_{1},b_{2}\}$. In other words, $G=P_{3}$, which
contradicts the hypothesis.

\emph{(ii)} $\Leftrightarrow$ \emph{(iii) }In \cite{StaplesThesis} it is shown
that for connected graphs \emph{(ii)} and \emph{(iii)} are equivalent.
Clearly, it can be relaxed to the condition that the graphs under
consideration have no isolated vertices.

\emph{(iii)} $\Rightarrow$ \emph{(i) }According to Theorem \ref{th444}, every
graph $G\in\mathbf{W}_{2}$ has the property that $G-v$ is well-covered, for
each $v\in V(G)$. In addition, $G\neq P_{3}$, since $P_{3}$ is even not well-covered.

\emph{(iii)} $\Rightarrow$ \emph{(iv)} Assume, to the contrary, that for some
non-maximum independent set $A$ in $G$ there is only one independent set, say
$B$, such that $A\cup B\in\Omega(G)$. Clearly, such a set $B$ must exist
because $G$ is well-covered, and we may suppose that $A\cap B=\emptyset$.
Since $G$ is in the class $\mathbf{W}_{2}$, it follows that there are
$S_{1},S_{2}\in\Omega(G)$, such that $A\subset S_{1},B\subset S_{2}$ and
$S_{1}\cap S_{2}=\emptyset$. Hence, $B\cap S_{1}=\emptyset$ which ensures that
$A$ can be extended to two maximum independent sets in $G$ by two disjoint
independent sets, namely, $B$ and $S_{1}-A$, in contradiction with the
assumption on $A$.

\emph{(iii)} $\Rightarrow$ \emph{(vi)} If $A$ is a non-maximum independent set
and $v\notin A$, then by definition of the class $\mathbf{W}_{2}$, it follows
that there are two disjoint maximum independent sets $S_{1},S_{2}$ in $G$,
such that $A\subset S_{1}$ and $\{v\}\subset S_{2}$. Clearly, $v\notin S_{1}$.

\emph{(iv)} $\Rightarrow$ \emph{(v) }It is clear.

\emph{(v)} $\Rightarrow$ \emph{(ii)} Evidently, $G$ is well-covered. Suppose,
to the contrary, that $G$ is not $1$-well-covered, i.e., there is some $v\in
V(G)$, such that $G-v$ is not well-covered. Hence, $v$ cannot be an isolated
vertex, and Lemma \ref{lem1} implies $\alpha(G-v)=\alpha(G)$. There exists
some maximal independent set $A$ in $G-v$, such that $\left\vert A\right\vert
<\alpha(G-v)$, because $G-v$ is not well-covered. Hence, for each $w\in
V(G)-\left(  A\cup\left\{  v\right\}  \right)  $ the set $A\cup\left\{
w\right\}  $ is not independent in $G-v$ and, consequently, in $G$. Therefore,
there is only one enlargement of $A$, namely $A\cup\left\{  v\right\}  $, to a
maximum independent set of $G$, in contradiction with the hypothesis.

\emph{(vi)} $\Rightarrow$ \emph{(vii) }It is evident.

\emph{(vii)} $\Rightarrow$ \emph{(ii)} Clearly, $G$ is well-covered. Assume,
to the contrary, that $G$ is not $1$-well-covered, i.e., there is some
$v_{0}\in V(G)$, such that $G-v_{0}$ is not well-covered. Since $v_{0}$ cannot
be isolated, Lemma \ref{lem1} implies $\alpha(G-v_{0})=\alpha(G)$. Further,
there exists some maximal independent set $A$ in $G-v_{0}$, with $\left\vert
A\right\vert <\alpha(G-v_{0})=\alpha(G)$. In other words, $A$ is a non-maximum
independent set in $G$ and $v_{0}\notin A$. By the hypothesis, there is a
maximum independent set $S$ in $G$, such that $A\subset S$ and $v_{0}\notin
S$. It follows that $S$ is an independent set in $G-v_{0}$, larger than $A$,
in contradiction to the maximality of $A$ in $G-v_{0}$. Therefore, $G$ must be
$1$-well-covered.
\end{proof}

We can now give alternative proofs for the following.

\begin{corollary}
\label{cor4444}Let $G$ be a graph belonging to $\mathbf{W}_{\mathbf{2}}$.

\emph{(i)} \cite{Pinter1991} For every non-maximum independent set $S$ in $G$,
the graph $G-N[S]$ is in the class $\mathbf{W}_{2}$ as well.

\emph{(ii)} \cite{Staples} If $G\neq K_{2}$\emph{ }is connected, then $G$ has
no leaf.
\end{corollary}

\begin{proof}
\emph{(i)} Let $S$ be a non-maximum independent set in $G$ and $A$ be a
non-maximum independent set in $G-N[S]$. Then $A\cup S$ is a non-maximum
independent set in $G$, and according to Theorem \ref{th1}\emph{(iv)}, there
exist two disjoint independent sets $S_{1},S_{2}$ in $G$ such that $A\cup
S\cup S_{1},A\cup S\cup S_{2}\in\Omega(G)$. Hence, $A\cup S_{1},A\cup S_{2}$
are maximum independent sets in $G-N[S]$. By Theorem \ref{th1}\emph{(iv)}, it
follows that $G-N[S]$ belongs to $\mathbf{W}_{2}$.

\emph{(ii)} Assume, to the contrary, that $G$ has a leaf, say $v$. Let
$N\left(  v\right)  =\left\{  u\right\}  $ and $w\in N\left(  u\right)
-\left\{  v\right\}  $. By Theorem \ref{th1}\emph{(vii)}, there exists some
$S\in\Omega\left(  G\right)  $, such that $\left\{  w\right\}  \subset S$ and
$v\notin S$. Hence, we infer that $S\cup\left\{  v\right\}  $ is independent,
contradicting the fact that $\left\vert S\cup\left\{  v\right\}  \right\vert
>\left\vert S\right\vert =\alpha\left(  G\right)  $.
\end{proof}

\begin{corollary}
\label{cor5}\cite{Pinter1991} If $G\in$ $\mathbf{W}_{\mathbf{2}}$, then
$G-N[v]\in\mathbf{W}_{\mathbf{2}}$, for each $v\in V(G)$.
\end{corollary}

The \textit{differential of a set} $A\subseteq V(G)$ is $\partial\left(
A\right)  $ $=\left\vert N(A)-A\right\vert -\left\vert A\right\vert $
\cite{Mashburn2006,BermudoFernau2012,Bermudo2014}. Clearly, if $S$ is an
independent set, then $\partial\left(  S\right)  =\left\vert N(S)\right\vert
-\left\vert S\right\vert $. The number
\[
\partial\left(  G\right)  =\max\left\{  \partial\left(  A\right)  :A\subseteq
V(G)\right\}
\]
is the \textit{differential of the graph} $G$. For instance, $\partial\left(
K_{p,q}\right)  =p+q-2$, while $\partial\left(  C_{7}\right)  =2$ and
$\partial\left(  C_{9}\right)  =3$.

\begin{theorem}
\label{prop11} If a connected\textbf{ }graph $G\neq K_{2}$ belongs to the
class $W_{2}$, then the following assertions hold:

\emph{(i)} for each $v\in V(G)$, there exist at least two disjoint sets
$S_{1},S_{2}\in\Omega\left(  G\right)  $ such that $v\notin S_{1}\cup S_{2}$;

\emph{(ii) }$G$ has at least $2\alpha(G)+1$ vertices;

\emph{(iii)} for every $u,v\in V(G)$, there is some $S\in\Omega(G)$, such that
$S\cap\{u,v\}=\emptyset$;

\emph{(iv)} $\alpha(G)\leq\mu(G)$ and $\alpha(G)+\mu(G)\leq\left\vert
V(G)\right\vert -1$;

\emph{(v)} $\alpha(G)=\alpha(G-S)$ holds for each independent set $S$;

\emph{(vi) }if\emph{ }$A\subseteq B\in$ \textrm{Ind}$\left(  G\right)  $, then
$\partial\left(  A\right)  \leq\partial\left(  B\right)  $; i.e., $\partial$
is monotonic over \textrm{Ind}$\left(  G\right)  $;

\emph{(vii) }$G$ is regularizable and $\left\vert B\right\vert <\left\vert
N(B)\right\vert $ for every independent set $B$;

\emph{(viii) }$\left\vert A\right\vert \leq\alpha\left(  G\left[  N\left(
A\right)  \right]  \right)  $ is true for every independent set $A$;

\emph{(ix)} for each $A\in$ \textrm{Ind}$\left(  G\right)  $ there is a
matching from $A$ into an independent set.
\end{theorem}

\begin{proof}
\emph{(i)} and \emph{(ii)} Let $v\in V(G)$. By Corollary \ref{cor4444}%
\emph{(ii)}, $\left\vert N(v)\right\vert \geq2$. Suppose $u,w\in N(v)$. Then
there are at least two disjoint maximum independent sets $S_{1},S_{2}$ in $G$
such that $u\in S_{1},w\in S_{2}$, because $G\in\mathbf{W}_{2}$. Since
$vu,vw\in E(G)$, it follows that $v\notin S_{1}\cup S_{2}$. Consequently, we
obtain
\[
1+2\alpha\left(  G\right)  =\left\vert \left\{  v\right\}  \cup S_{1}\cup
S_{2}\right\vert \leq\left\vert V\left(  G\right)  \right\vert ,
\]
as claimed.

\emph{(iii) }Let $u,v\in V(G)$. By Part \emph{(i)}, there are two disjoint
maximum independent sets $S_{1},S_{2}$ in $G$ such that $v\notin S_{1}\cup
S_{2}$. Hence, $u$ belongs to at most one of $S_{1},S_{2}$, say to $S_{1}$.
Therefore, $S_{2}\cap\{u,v\}=\emptyset$.

\emph{(iv) }Let $v\in V(G)$. According to Part \emph{(i)}, there are at least
two disjoint maximum independent sets $S_{1},S_{2}$ in $G$ such that $v\notin
S_{1}\cup S_{2}$. By Theorem \ref{th6}, there is a perfect matching $M$ in
$H=G[S_{1}\cup S_{2}]$. Therefore, $\alpha\left(  G\right)  =\left\vert
M\right\vert \leq\mu\left(  G\right)  $. By Part \emph{(ii)}, we have
$\alpha(G)\leq(\left\vert V(G)\right\vert -1)/2$. Since $\mu\left(  G\right)
\leq\left\vert V(G)\right\vert /2$, we obtain
\[
\alpha(G)+\mu\left(  G\right)  \leq(\left\vert V(G)\right\vert
-1)/2+\left\vert V(G)\right\vert /2=\left\vert V(G)\right\vert -1/2,
\]
which means that $\alpha(G)+\mu\left(  G\right)  \leq\left\vert
V(G)\right\vert -1$.

\emph{(v)} Let $S$ be an independent set in $G$ and $v\in V(G)-S$. Since $G\in
W_{2}$, there exist two disjoint maximum independent sets $S_{1},S_{2}$ in $G$
such that $S\subseteq S_{1}$ and $v\in S_{2}$. Hence, $S_{2}\subseteq V(G)-S$
and this implies that $\left\vert S_{2}\right\vert \leq\alpha(G-S)\leq
\alpha(G)$, i.e., $\alpha(G)=\alpha(G-S)$.

\emph{(vi) }The sets\emph{ }$A$ and $B-A$ are independent and disjoint. Then,
by definition of the class $W_{2}$, there is a maximum independent set $S$
including $A$ such that $S\cap\left(  B-A\right)  =\emptyset$. Hence,
$\left\vert N(A)\right\vert \leq\left\vert N(B)\right\vert -\left\vert S\cap
N(B)\right\vert $. By Berge's theorem there is a matching from $B-A$ into
$S-A$. It means that
\[
\left\vert S\cap N(B)\right\vert =\left\vert S\cap N(B-A)\right\vert
\geq\left\vert B-A\right\vert =\left\vert B\right\vert -\left\vert
A\right\vert \text{.}%
\]
Therefore,
\[
\left\vert N(A)\right\vert \leq\left\vert N(B)\right\vert -\left\vert S\cap
N(B)\right\vert \leq\left\vert N(B)\right\vert -\left(  \left\vert
B\right\vert -\left\vert A\right\vert \right)  \text{,}%
\]
which concludes with
\[
\partial\left(  A\right)  =\left\vert N(A)\right\vert -\left\vert A\right\vert
\leq\left\vert N(B)\right\vert -\left\vert B\right\vert =\partial\left(
B\right)  \text{.}%
\]

\emph{(vii) }If $G=K_{2}$, then $G$ is regularizable, according to Theorem
\ref{th11}\emph{(ii)}.

If $G\neq K_{2}$, then Corollary \ref{cor6} ensures that $G$ is not bipartite.
Suppose $B$ is an independent set and $v\in B$, i.e., $\{v\}\subseteq B$.
Hence, using Part \emph{(vi)}, we obtain%
\[
\deg\left(  v\right)  -1=\left\vert N(v)\right\vert -1\leq\left\vert
N(B)\right\vert -\left\vert B\right\vert \text{.}%
\]
Thus $\left\vert B\right\vert <\left\vert N(B)\right\vert $, since
$\deg\left(  v\right)  \geq2$, in accordance with Corollary \ref{cor4444}%
\emph{(ii)}. Finally, by Theorem \ref{th11}\emph{(i)}, $G$ is regularizable.

\emph{(viii) }Assume, to the contrary, that $\left\vert A\right\vert
>\alpha\left(  G\left[  N\left(  A\right)  \right]  \right)  $ for some
independent set $A$. Let $B$ be a maximum independent set in $G\left[
N\left(  A\right)  \right]  $. By Theorem \ref{th1}\emph{(vi)} there exists
$S\in\Omega(G)$ such that $B\subset S$ and $A\cap S=\emptyset$. Since $\left(
N\left(  A\right)  -B\right)  \cap\left(  S-B\right)  =\emptyset$, we infer
that $A\cup\left(  S-B\right)  $ is independent. Finally,
\[
\left\vert A\cup\left(  S-B\right)  \right\vert =\left\vert A\right\vert
+\left\vert S-B\right\vert >\left\vert S\right\vert =\alpha(G),
\]
which is a contradiction.

\emph{(ix)} Let $A$ and $B$ be an independent sets such that $B\subseteq
N\left(  A\right)  $. Since $G\in\mathbf{W}_{\mathbf{2}}$, there exist
disjoint maximum independent sets $S_{1}$, $S_{2}$ such that $A\subseteq
S_{1}$ and $B\subseteq S_{2}$. By Theorem \ref{th6}, there is a matching from
$S_{1}$ to $S_{2}$. Thus $A$ is matched into an independent set included in
$S_{2}\cap N\left(  A\right)  $.
\end{proof}

It is worth mentioning that there are graphs not in class $W_{2}$, that
satisfy Theorem \ref{prop11}; e.g., the graph $C_{7}$.

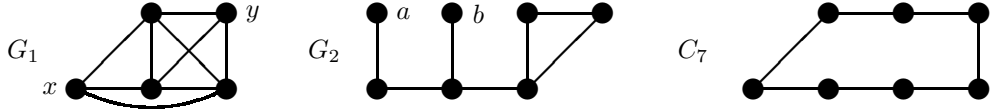
\begin{figure}[h]
\setlength{\unitlength}{1cm}\begin{picture}(5,1.15)\thicklines
\multiput(1,0)(1,0){3}{\circle*{0.29}}
\multiput(2,1)(1,0){2}{\circle*{0.29}}
\put(1,0){\line(1,0){2}}
\put(1,0){\line(1,1){1}}
\put(2,0){\line(0,1){1}}
\put(2,0){\line(1,1){1}}
\put(2,1){\line(1,0){1}}
\put(2,1){\line(1,-1){1}}
\put(3,0){\line(0,1){1}}
\qbezier(1,0)(2,-0.5)(3,0)
\put(0.65,0){\makebox(0,0){$x$}}
\put(3.35,1){\makebox(0,0){$y$}}
\put(0.3,0.5){\makebox(0,0){$G_{1}$}}
\multiput(5,0)(1,0){3}{\circle*{0.29}}
\multiput(5,1)(1,0){4}{\circle*{0.29}}
\put(5,0){\line(1,0){2}}
\put(5,0){\line(0,1){1}}
\put(6,0){\line(0,1){1}}
\put(7,1){\line(1,0){1}}
\put(7,0){\line(1,1){1}}
\put(7,0){\line(0,1){1}}
\put(5.35,1){\makebox(0,0){$a$}}
\put(6.35,1){\makebox(0,0){$b$}}
\put(4.3,0.5){\makebox(0,0){$G_{2}$}}
\multiput(10,0)(1,0){4}{\circle*{0.29}}
\multiput(11,1)(1,0){3}{\circle*{0.29}}
\put(10,0){\line(1,0){3}}
\put(10,0){\line(1,1){1}}
\put(11,1){\line(1,0){2}}
\put(13,0){\line(0,1){1}}
\put(9.2,0.5){\makebox(0,0){$C_{7}$}}
\end{picture}\caption{$\left\vert \left\{  x,y\right\}  \right\vert
>\alpha\left(  G_{1}\left[  N\left(  \left\{  x,y\right\}  \right)  \right]
\right)  $ and $\left\vert \left\{  a,b\right\}  \right\vert >\alpha\left(
G_{2}\left[  N\left(  \left\{  a,b\right\}  \right)  \right]  \right)  $.}%
\label{fig1}%
\end{figure}

Neither quasi-regularizable graphs nor well-covered graphs have to satisfy
Theorem \ref{prop11}\emph{(viii)}; e.g., the graphs $G_{1}$ and $G_{2}$ from
Figure \ref{fig1}, respectively.

Actually, Theorem \ref{prop11}\emph{(vi) }is a generalization of the following.

\begin{corollary}
\cite{Staples}\label{th44} If $S$ is an independent set in a connected graph
$G$ belonging the class $\mathbf{W}_{\mathbf{2}}$, then $\deg\left(  v\right)
\leq\left\vert N(S)\right\vert -\left\vert S\right\vert +1$ for every $v\in S$.
\end{corollary}

There are some known lower bounds on $\partial\left(  G\right)  $
\cite{BermudoFernau2012,Bermudo2014}. Here we give a new one for connected
$1$-well-covered graphs.

\begin{corollary}
\label{cor4}If $G\in\mathbf{W}_{\mathbf{2}}$, then $\partial\left(  G\right)
\geq\left\vert V(G)\right\vert -2\alpha(G)\geq\Delta(G)-1$.
\end{corollary}

\begin{proof}
Clearly, $\partial\left(  G\right)  =\max\left\{  \partial\left(  A\right)
:A\subseteq V(G)\right\}  \geq\max\left\{  \partial\left(  S\right)  :S\text{
is independent}\right\}  $. By Theorem \ref{prop11}\emph{(vi)}, we know that
\begin{align*}
\max\left\{  \partial\left(  S\right)  :S\text{ is independent}\right\}   &
=\max\left\{  \partial\left(  S\right)  :S\in\Omega\left(  G\right)  \right\}
\\
&  =\left\vert V(G)\right\vert -2\alpha(G).
\end{align*}
Finally, taking $v\in V\left(  G\right)  $ with $\deg v=\Delta(G)$ and
$S\in\Omega\left(  G\right)  $ be such that $v\in S$, Corollary \ref{th44}
gives
\[
\Delta(G)=\deg\left(  v\right)  \leq\left\vert N(S)\right\vert -\left\vert
S\right\vert +1=\left\vert V\left(  G\right)  -S\right\vert -\left\vert
S\right\vert +1=\left\vert V(G)\right\vert -2\alpha(G)+1,
\]
as required.
\end{proof}

Notice that there are graphs not in $\mathbf{W}_{\mathbf{2}}$ that enjoy the
conclusions from Corollaries \ref{th44}, \ref{cor4}; e.g., the cycle $C_{9}$,
which is not even well-covered.

Evidently, $G\in$ $\mathbf{W}_{\mathbf{2}}$ if and only if each of its
connected components belongs to $\mathbf{W}_{\mathbf{2}}$. In addition, it is
easy to see that:

\begin{itemize}
\item every graph $G=nK_{2},n\geq1$, is in class $\mathbf{W}_{2}$, and has
exactly $2\alpha(G)$ vertices;

\item each graph $G\in$ $\left\{  C_{5}\cup nK_{2},C_{3}\cup nK_{2}%
:n\geq1\right\}  $ belongs to $\mathbf{W}_{2}$ and has exactly $2\alpha(G)+1$ vertices.
\end{itemize}

\begin{corollary}
\label{cor6}\emph{(i)} $K_{2}$ is the unique connected graph in $\mathbf{W}%
_{2}$ of order $2\alpha(G)$.

\emph{(ii)} $C_{3}$ and $C_{5}$ are the only two connected graphs in
$\mathbf{W}_{2}$ of order $2\alpha(G)+1$.

\emph{(iii)} $K_{2}$ is the only connected bipartite graph belonging to
$\mathbf{W}_{2}$.
\end{corollary}

\begin{proof}
\emph{(i)} On the one hand, according to Theorem \ref{prop11}\emph{(ii)}, we
have that $2\alpha(G)+1\leq\left\vert V(G)\right\vert $, whenever
$G\in\mathbf{W}_{2}$ is connected and $G\neq K_{2}$. On the other hand,
$K_{2}$ belongs to $\mathbf{W}_{\mathbf{2}}$ and $2\alpha(K_{2})=2=\left\vert
V(K_{2})\right\vert $, and hence the conclusion follows.

\emph{(ii)} Let $G$ be a connected graph in $\mathbf{W}_{2}$ of order
$2\alpha(G)+1$. By Corollary \ref{th44}\emph{(ii)}, we have that
$\Delta(G)\leq\left\vert V(G)\right\vert -2\alpha(G)+1=2$.

If $\Delta(G)\leq1$, then $G\in\left\{  K_{1},K_{2}\right\}  $ and this
contradicts $\left\vert V(G)\right\vert =2\alpha(G)+1$.

If $\Delta(G)=2$, then $G\neq$ $K_{2}$ and, according to Corollary
\ref{cor4444}\emph{(ii)}, we infer that the degree of every vertex in $G$ is
equal to $2$. Since $G$ is connected, Theorem \ref{th1}\emph{(i)} implies that
$G\in\left\{  C_{3};C_{5}\right\}  $.

\emph{(iii)} Let $G$ $\in\mathbf{W}_{2}$ be a connected bipartite graph,
having $\left\{  A,B\right\}  $ as its bipartition. Hence, there exist
$S_{1},S_{2}$ disjoint maximum independent sets such that $A\subseteq S_{1}$
and $B\subseteq S_{2}$, because $A,B$ are disjoint and independent. Since
$S_{1}\cap B=\emptyset=S_{2}\cap A$, we infer that $A=S_{1}$ and $B=S_{2}$.
Hence, $\left\vert V\left(  G\right)  \right\vert =\left\vert A\cup
B\right\vert =2\alpha\left(  G\right)  $. Consequently, $G=K_{2}$, because,
otherwise, by Theorem \ref{prop11}\emph{(ii)}, $G$ must have at least
$2\alpha(G)+1$ vertices.
\end{proof}

\section{A way from $\mathbf{W}_{1}$ to $\mathbf{W}_{2}$}

Let $v\in V\left(  G\right)  $. If for every independent set $S$ of
$G-N\left[  v\right]  $, there exists some $u\in N\left(  v\right)  $ such
that $S\cup\left\{  u\right\}  $ is independent, then $v$ is a
\textit{shedding vertex} of $G$ \cite{Woodroofe2009}. Let $Shed\left(
G\right)  $ denote the set of all shedding vertices. For instance,
$Shed\left(  P_{4}\right)  =\left\{  v:\deg(v)=2\right\}  $, $Shed\left(
C_{4}\right)  =Shed\left(  C_{k}\right)  =\emptyset,k\geq6$, while
$Shed\left(  C_{3}\right)  =V\left(  C_{3}\right)  $, $Shed\left(
C_{5}\right)  =V\left(  C_{5}\right)  $. Clearly, no isolated vertex may be a
shedding vertex. On the other hand, every vertex $v\in V\left(  G\right)  $ of
degree $\left\vert V\left(  G\right)  \right\vert -1$ is a shedding vertex.

Let us define $\varepsilon_{G}:\mathrm{Ind}(G)\longrightarrow%
\mathbb{N}
$ as $\varepsilon_{G}\left(  A\right)  =\max\left\{  \left\vert S\right\vert
:A\subseteq S\text{ and }S\in\mathrm{Ind}(G)\right\}  $. Informally, one can
interpret the function $\varepsilon_{G}$ as the strength of enlargement of
independent sets. The following basic properties of the function
$\varepsilon_{G}$ are clear.

\begin{lemma}
\label{lem2}It is true for every graph $G$ that:

\emph{(i)} if $A\in\mathrm{Ind}(G)$, then $\left\vert A\right\vert
\leq\varepsilon_{G}\left(  A\right)  \leq\alpha\left(  G\right)  $;

\emph{(ii)} if $A\subseteq B\in\mathrm{Ind}(G)$, then $\varepsilon_{G}\left(
A\right)  \geq\varepsilon_{G}\left(  B\right)  $;

\emph{(iii)} if $H$ is an induced subgraph of $G$, then $\varepsilon
_{G}\left(  A\right)  \geq\varepsilon_{H}\left(  A\right)  $ for each
$A\in\mathrm{Ind}(H)$;

\emph{(iv)} $G$ is well-covered if and only if $\varepsilon_{G}\left(
A\right)  =\alpha\left(  G\right)  $ for every $A\in\mathrm{Ind}(G)$.
\end{lemma}

There is a natural connection between shedding vertices and the function
$\varepsilon_{G}$.

\begin{theorem}
\label{th10}Let $v\in$ $V\left(  G\right)  $. Then $v\in Shed\left(  G\right)
$ if and only if $\varepsilon_{G-v}\left(  A\right)  =\varepsilon_{G}\left(
A\right)  $ for each $A\in\mathrm{Ind}(G-v)$.
\end{theorem}

\begin{proof}
\textquotedblleft\textit{If\textquotedblright} Suppose that $\varepsilon
_{G-v}\left(  A\right)  =\varepsilon_{G}\left(  A\right)  $ for each
$A\in\mathrm{Ind}(G-v)$.

Assume that $B$ is an independent set of $G-N\left[  v\right]  $. Let us
choose an independent $S$ of size $\varepsilon_{G-v}\left(  B\right)  $ such
that $B\subseteq S\subseteq V\left(  G\right)  -v$. It follows that $S\cap
N\left(  v\right)  \neq\emptyset$, otherwise $S\cup\left\{  v\right\}  $ is
independent and $\varepsilon_{G-v}\left(  B\right)  =\left\vert S\right\vert
<\left\vert S\right\vert +1=\varepsilon_{G}\left(  B\right)  $ in
contradiction with the hypothesis. Finally, we conclude that there exists a
vertex $u\in S\cap N\left(  v\right)  \subseteq N\left(  v\right)  $ such that
$B\cup\left\{  u\right\}  $ is an independent set. Thus $v\in Shed\left(
G\right)  $, as claimed.

\textquotedblleft\textit{Only if\textquotedblright} Let $v$ be a shedding
vertex and $A$ be an independent set in $G-v$. By Lemma \ref{lem2}%
\emph{(iii)}, it is enough to prove that $\varepsilon_{G-v}\left(  A\right)
\geq\varepsilon_{G}\left(  A\right)  $.

\qquad\textit{Case 1}. $A\cap N\left(  v\right)  \neq\emptyset$. It implies
that no independent set in $G$ including $A$ contains $v$. Thus $\varepsilon
_{G-v}\left(  A\right)  =\varepsilon_{G}\left(  A\right)  $.

\qquad\textit{Case 2}. $A\cap N\left(  v\right)  =\emptyset$. It means that
$A$ $\subseteq V\left(  G\right)  -N_{G}\left[  v\right]  $. The structure of
a maximum independent enlargement of $A$ in $G$ is a union of $A$, an
independent subset $B\subseteq V\left(  G\right)  -N_{G}\left[  v\right]  $,
and a vertex, say $x$, belonging to $N_{G}\left[  v\right]  $. If $x\neq v$,
then $\varepsilon_{G-v}\left(  A\right)  =\varepsilon_{G}\left(  A\right)  $.
Otherwise, $\left\vert A\cup B\cup\left\{  v\right\}  \right\vert
=\varepsilon_{G}\left(  A\right)  $. Since $A\cup B\subseteq V\left(
G\right)  -N_{G}\left[  v\right]  $ is independent and $v$ is shedding, we
conclude that there exists $w\in N_{G}\left(  v\right)  $ such that $A\cup
B\cup\left\{  w\right\}  $ is independent in $G-v$. Hence,
\[
\varepsilon_{G-v}\left(  A\right)  \geq\left\vert A\cup B\cup\left\{
w\right\}  \right\vert =\left\vert A\cup B\cup\left\{  v\right\}  \right\vert
=\varepsilon_{G}\left(  A\right)  \text{,}%
\]
which completes the proof.
\end{proof}

In well-covered graphs shedding vertices may be characterized in more specific manner.

\begin{corollary}
\label{th3}Let $v$ be a non-isolated vertex of a well-covered graph $G$. Then
$v\in Shed\left(  G\right)  $ if and only if $G-v$ is well-covered.
\end{corollary}

\begin{proof}
Lemma \ref{lem1} claims that $\alpha\left(  G\right)  =\alpha\left(
G-v\right)  $. In accordance with Lemma \ref{lem2}\emph{(iv)}, $G$ is
well-covered if and only if $\varepsilon_{G}%
\equiv
\alpha\left(  G\right)  $.

\textquotedblleft\textit{If\textquotedblright} Suppose $G-v$ is well-covered.
Then
\[
\varepsilon_{G-v}%
\equiv
\alpha\left(  G-v\right)  =\alpha\left(  G\right)
\equiv
\varepsilon_{G}%
\]
Consequently, by Theorem \ref{th10}, $v\in Shed\left(  G\right)  $.

\textquotedblleft\textit{Only if\textquotedblright} Let $v\in Shed\left(
G\right)  $. By Theorem \ref{th10},%
\[
\varepsilon_{G-v}%
\equiv
\varepsilon_{G}%
\equiv
\alpha\left(  G\right)  =\alpha\left(  G-v\right)
\]
In conclusion, $G-v$ is well-covered.
\end{proof}

Notice that $P_{3}$ is not a well-covered graph, while $P_{3}-v$ is
well-covered, for each $v\in V(P_{3})$, while $\left\vert Shed\left(
P_{3}\right)  \right\vert =1$.

\begin{corollary}
\label{cor7}Let $G$ be a well-covered graph and $v\in V(G)$ is non-isolated.
The following conditions are equivalent:

\emph{(i)} $G-v$ is well-covered;

\emph{(ii)} $\left\vert N_{G}(v)-N_{G}(S)\right\vert \geq1$ for every
independent set $S$ of $G-N_{G}[v]$;

\emph{(iii)} there is no independent set $S\subseteq V(G)-N_{G}[v]$ such that
$v$ is isolated in $G-N_{G}[S]$;

\emph{(iv)} $v$ is a shedding vertex.
\end{corollary}

\begin{proof}
The equivalences \emph{(i)} $\Leftrightarrow$ \emph{(ii) }$\Leftrightarrow$
\emph{(iii)} were established in \cite{FinHarNow}, while \emph{(ii)}
$\Leftrightarrow$ \emph{(iv)} appears in \cite{CaCrRey2016}. Corollary
\ref{th3} gives an alternative proof for \emph{(i)} $\Leftrightarrow$
\emph{(iv)}.
\end{proof}

A vertex $v$ of a graph $G$ is \textit{simplicial} if the induced subgraph of
$G$ on the set $N[v]$ is a complete graph and this complete graph is called a
simplex of $G$. Clearly, every leaf is a simplicial vertex. Let $Simp\left(
G\right)  $ denote the set of all simplicial vertices. For instance, if
$n\geq4$, then $Simp\left(  C_{n}\right)  =\emptyset$, while $Simp\left(
K_{n}\right)  =V\left(  K_{n}\right)  $. A graph $G$ is said to be
\textit{simplicial} if every vertex of $G$ belongs to a simplex of $G$. For
example, $P_{n}$ is simplicial only for $n\leq4$.

\begin{theorem}
\cite{PrToppVest1996}\label{th2}\ A graph $G$ is simplicial and well-covered
if and only if every vertex of $G$ belongs to exactly one simplex.
\end{theorem}

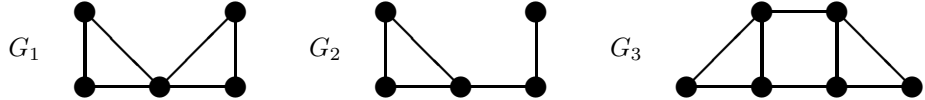
\begin{figure}[h]
\setlength{\unitlength}{1cm}\begin{picture}(5,1.15)\thicklines
\multiput(1.5,0)(1,0){3}{\circle*{0.29}}
\multiput(1.5,1)(2,0){2}{\circle*{0.29}}
\put(1.5,0){\line(1,0){2}}
\put(1.5,1){\line(1,-1){1}}
\put(1.5,0){\line(0,1){1}}
\put(2.5,0){\line(1,1){1}}
\put(3.5,0){\line(0,1){1}}
\put(0.7,0.5){\makebox(0,0){$G_{1}$}}
\multiput(5.5,0)(1,0){3}{\circle*{0.29}}
\multiput(5.5,1)(2,0){2}{\circle*{0.29}}
\put(5.5,0){\line(1,0){2}}
\put(5.5,0){\line(0,1){1}}
\put(5.5,1){\line(1,-1){1}}
\put(7.5,0){\line(0,1){1}}
\put(4.7,0.5){\makebox(0,0){$G_{2}$}}
\multiput(9.5,0)(1,0){4}{\circle*{0.29}}
\multiput(10.5,1)(1,0){2}{\circle*{0.29}}
\put(9.5,0){\line(1,0){3}}
\put(9.5,0){\line(1,1){1}}
\put(10.5,0){\line(0,1){1}}
\put(10.5,1){\line(1,0){1}}
\put(11.5,0){\line(0,1){1}}
\put(11.5,1){\line(1,-1){1}}
\put(8.7,0.5){\makebox(0,0){$G_{3}$}}
\end{picture}\caption{Simplicial graphs. Only $G_{1}$ is not well-covered.
$G_{3}$ is in $\mathbf{W}_{2}$.}%
\label{fig3}%
\end{figure}

\begin{proposition}
\label{prop1}\cite{Woodroofe2009} If $v\in Simp\left(  G\right)  $, then
$N\left(  v\right)  \subseteq Shed\left(  G\right)  $.
\end{proposition}

Corollary \ref{th3} and Proposition \ref{prop1} imply the following.

\begin{corollary}
\cite{Baker}\label{cor3} If $G$ is a well-covered graph and $v\in Simp\left(
G\right)  $, then $G-u$ is well-covered for each $u\in N\left(  v\right)  $.
\end{corollary}

\begin{proposition}
\label{prop4}If each vertex of $G$ belongs to exactly one simplex and every
simplex contains at least two simplicial vertices, then $G$ is in
$\mathbf{W}_{\mathbf{2}}$.
\end{proposition}

\begin{proof}
By Theorem \ref{th2}, $G$ is well-covered. Further, Corollary \ref{cor3}
ensures that $G-v$ is well-covered for each $v\in V\left(  G\right)  $.
Consequently, $G$ belongs to $\mathbf{W}_{\mathbf{2}}$, according to Theorem
\ref{th1}\emph{(i)}, because, clearly, $G\neq P_{3}$.
\end{proof}

There are simplicial graphs in $\mathbf{W}_{\mathbf{2}}$, which do not satisfy
the condition that every simplex must contain at least two simplicial
vertices; e.g., consider the graph $G_{3}$ from Figure \ref{fig3}. Notice that
if $Shed\left(  G\right)  =V\left(  G\right)  $, it is not true that $G$
belongs to $\mathbf{W}_{\mathbf{2}}$; e.g., the graph $G_{1}$ from Figure
\ref{fig3}.

\begin{theorem}
\label{th12}Let $G$\ be a well-covered graph without isolated vertices. Then
the following assertions are equivalent:

\emph{(i)} $G$\emph{ }belongs to the class $\mathbf{W}_{\mathbf{2}}$;

\emph{(ii)} the differential function is monotonic over $\mathrm{Ind}(G)$,
i.e.,\emph{ }if\emph{ }$A\subseteq B\in\mathrm{Ind}\left(  G\right)  $, then
$\partial\left(  A\right)  \leq\partial\left(  B\right)  $;

\emph{(iii) }$Shed\left(  G\right)  =V\left(  G\right)  $;

\emph{(iv)} no independent set $S$ leaves an isolated vertex in $G-N_{G}[S]$;

\emph{(v)} $G-N_{G}\left[  v\right]  \in\mathbf{W}_{\mathbf{2}}$ for every
$v\in V\left(  G\right)  $.
\end{theorem}

\begin{proof}
\emph{(i) }$\Leftrightarrow$ \emph{(ii)} \textquotedblleft\textit{If}%
\textquotedblright\ Let $A$ be a non-maximum independent set and $v\notin A$.
By Theorem \ref{th1}\emph{(vii)},\emph{ }it is enough to find some $S\in
\Omega(G)$ such that $A\subset S$ and $v\notin S$.

\qquad\textit{Case 1.} $v\in N\left(  A\right)  $. Since $G$\ is well-covered,
there exists a maximum independent set including $A$, say $S$. Clearly,
$v\notin S$.

\qquad\textit{Case 2.} $v\notin N\left(  A\right)  $. Hence, $B=A\cup\{v\}$ is
independent. By the monotonicity property,
\[
\left\vert N(A)\right\vert -\left\vert A\right\vert =\partial\left(  A\right)
\leq\partial\left(  B\right)  =\left\vert N(B)\right\vert -\left\vert
B\right\vert =\left\vert N(A\cup\{v\})\right\vert -\left\vert A\right\vert
-1\text{.}%
\]
Thus, $\left\vert N(A)\right\vert +1\leq\left\vert N(A\cup\{v\})\right\vert $,
which means that there is $w\in N(v)-N(A)$. Since $G$\ is well-covered, there
exists a maximum independent set including $A\cup\{w\}$, say $S$. Clearly,
$v\notin S$.

\textquotedblleft\textit{Only if}\textquotedblright\ It follows from Theorem
\ref{prop11}\emph{(vi).}

\emph{(i) }$\Leftrightarrow$ \emph{(iii)} Clearly, $G\neq P_{3}$, because
$P_{3}$ is not well-covered. The rest of the proof is in Corollary \ref{th3}
and Theorem \ref{th1}\emph{(i).}

\emph{(i) }$\Leftrightarrow$ \emph{(iv)} It follows directly from Corollary
\ref{cor7}\emph{(iii)}.

\emph{(i) }$\Rightarrow$ \emph{(v)}\ \ It follows from Corollary \ref{cor5}

\emph{(v) }$\Rightarrow$ \emph{(iv) }Let $S$ be a non-maximum independent set
in $G$. Since $G$ is well-covered, there exists some maximum independent set
$A$ such that $S\subset A$. Let $v\in A-S$ and $u\in S$. Clearly, $S\cap
N_{G}\left[  v\right]  =\emptyset$, and $G-N_{G}\left[  u\right]
\in\mathbf{W}_{\mathbf{2}}$, by the hypothesis. Hence, according to Corollary
\ref{cor4444}\emph{(i)}, it follows that the graph
\[
\left(  G-N_{G}\left[  u\right]  \right)  -N_{G-N_{G}\left[  u\right]
}\left[  S-u\right]  =G-N_{G}\left[  S\right]
\]
belongs to $\mathbf{W}_{\mathbf{2}}$ as well. Hence, $N_{G}\left(  v\right)
\nsubseteqq$ $N_{G}\left(  S\right)  $, otherwise $v$ is isolated in
$G-N_{G}\left[  S\right]  $. In addition, no $w\in N_{G}\left(  v\right)  $ is
isolated in $G-N_{G}\left[  S\right]  $, because $wv\in E\left(  G\right)  $.
Finally, no $w$ from $V\left(  G-N_{G}\left[  v\right]  \right)  -N_{G}\left[
S\right]  $ is isolated in $G-N_{G}\left[  S\right]  $, since $G-N_{G}\left[
v\right]  \in\mathbf{W}_{\mathbf{2}}$.
\end{proof}

\section{Graph operations}

In \cite{Staples} are shown a number of ways to build graphs in class
$\mathbf{W}_{n}$, using graphs from $\mathbf{W}_{n}$ or $\mathbf{W}_{n+1}$. In
the following we make known how to create infinite subfamilies of
$\mathbf{W}_{\mathbf{2}}$, by means of corona, join, and concatenation of graphs.

Let $\mathcal{H}=\{H_{v}:v\in V(G)\}$ be a family of graphs indexed by the
vertex set of a graph $G$. The corona $G\circ\mathcal{H}$ of $G$ and
$\mathcal{H}$ is the disjoint union of $G$ and $H_{v},v\in V(G)$, with
additional edges joining each vertex $v\in V(G)$ to all the vertices of
$H_{v}$. If $H_{v}=H$ for every $v\in V(G)$, then we denote $G\circ H$ instead
of $G\circ\mathcal{H}$ \cite{FruchtHarary}.

Recall that the \textit{girth} of a graph $G$ is the length of a shortest
cycle contained in $G$, and it is defined as the infinity for every forest.

\begin{theorem}
\label{th4}\emph{(i)} \cite{FinHarNow} Let $G$ be a connected graph of girth
$\geq6$, which is isomorphic to neither $C_{7}$ nor $K_{1}$. Then $G$ is
well-covered if and only if $G=H\circ K_{1}$ for some graph $H$.

\emph{(ii)} \cite{LevMan2007} Let $G$ be a connected graph of girth $\geq5$.
Then $G$ is very well-covered if and only if $G=H\circ K_{1}$ for some graph
$H$.
\end{theorem}

Using corona operation one can build well-covered graphs of any girth as follows.

\begin{proposition}
\label{prop3}\cite{ToppLutz} The corona $G\circ\mathcal{H}$ of $G$ and
$\mathcal{H}=\{H_{v}:v\in V(G)\}$ is well-covered if and only if each
$H_{v}\in\mathcal{H}$ is a complete graph on at least one vertex.
\end{proposition}

For example, all the graphs in Figure \ref{fig23} are of the form
$G\circ\mathcal{H}$, but only $G_{1}$ is not well-covered, while $G_{3}$ is
$1$-well-covered.

\begin{figure}[h]
\setlength{\unitlength}{1cm}\begin{picture}(5,1.25)\thicklines
\multiput(1.5,0)(1,0){2}{\circle*{0.29}}
\multiput(1.5,1)(1,0){3}{\circle*{0.29}}
\put(1.5,0){\line(1,0){1}}
\put(1.5,0){\line(0,1){1}}
\put(2.5,0){\line(0,1){1}}
\put(2.5,0){\line(1,1){1}}
\put(0.8,0.5){\makebox(0,0){$G_{1}$}}
\multiput(5,0)(1,0){2}{\circle*{0.29}}
\multiput(5,1)(1,0){3}{\circle*{0.29}}
\put(5,0){\line(1,0){1}}
\put(5,0){\line(0,1){1}}
\put(6,0){\line(0,1){1}}
\put(6,1){\line(1,0){1}}
\put(6,0){\line(1,1){1}}
\put(4.3,0.5){\makebox(0,0){$G_{2}$}}
\multiput(9.5,0)(1,0){2}{\circle*{0.29}}
\multiput(8.5,1)(1,0){5}{\circle*{0.29}}
\put(8.5,1){\line(1,0){1}}
\put(8.5,1){\line(1,-1){1}}
\put(9.5,0){\line(1,0){1}}
\put(9.5,0){\line(0,1){1}}
\put(10.5,0){\line(1,1){1}}
\put(10.5,0){\line(0,1){1}}
\put(10.5,1){\line(1,0){2}}
\put(10.5,0){\line(2,1){2}}
\qbezier(10.5,1)(11.5,1.5)(12.5,1)
\put(7.8,0.5){\makebox(0,0){$G_{3}$}}
\end{picture}\caption{$G_{1}=P_{2}\circ\left\{  K_{1},2K_{1}\right\}  $,
$G_{2}=P_{2}\circ\left\{  K_{1},K_{2}\right\}  $, $G_{3}=P_{2}\circ\left\{
K_{2},K_{3}\right\}  $.}%
\label{fig23}%
\end{figure}
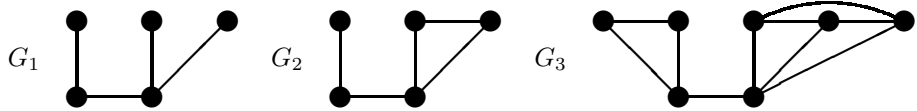

\begin{proposition}
\label{prop2} Let $L=G\circ\mathcal{H}$, where $\mathcal{H}=\{H_{v}:v\in
V(G)\}$ and $G$ is an arbitrary graph. Then $L$ belongs to $\mathbf{W}_{2}$ if
and only if each $H_{v}\in\mathcal{H}$ is a complete graph of order two at
least, for every non-isolated vertex $v$, while for each isolated vertex $u$,
its corresponding $H_{u}$ may be any complete graph.
\end{proposition}

\begin{proof}
Suppose that $L\in\mathbf{W}_{2}$. Then $L$ is well-covered, and therefore
each $H_{v}\in\mathcal{H}$ is a complete graph on at least one vertex, by
Proposition \ref{prop3}. Assume that for some non-isolated vertex $a\in V(G)$
its corresponding $H_{a}=K_{1}=(\{b\},\emptyset)$. Let $c\in N_{G}(a)$ and $B$
be a non-maximum independent set in $L$ containing $c$. Since $\alpha
(L)=\left\vert V(G)\right\vert $, it follows that every maximum independent
set $S$ of $L$ that includes $B$ must contain the vertex $b$. In other words,
$L$ could not be in $\mathbf{W}_{2}$, according to Theorem \ref{th1}%
\emph{(vi)}. Therefore, each $H_{v}\in\mathcal{H}$ must be a complete graph on
at least two vertices.

Conversely, if each $H_{v}\in\mathcal{H}$ is a complete graph on at least two
vertices, then $L$ is well-covered, by Proposition \ref{prop3}. Let $A$ be a
non-maximum independent set in $L$, and some vertex $b\notin A$. Since $L$ is
well-covered, there is some maximum independent set $S_{1}$ in $L$ such that
$A\subset S_{1}$. If $b\in S_{1}$, let $a\in N_{L}(b)-V(G)$. Hence
$S_{2}=S_{1}\cup\{a\}-\{b\}$ is a maximum independent set in $L$ with
$A\subset S_{2}$. In other words, there is a maximum independent set in $L$,
namely $S\in\{S_{1},S_{2}\}$, such that $A\subset S$ and $b\notin S$.
Therefore, according to Theorem \ref{th1}\emph{(v)}, it follows that
$L\in\mathbf{W}_{2}$. Clearly, if $v$ is isolated in $G$, then even
$H_{v}=K_{1}$ ensures $L$ to be in $\mathbf{W}_{2}$.
\end{proof}

If $\mathcal{H}=\{H_{v}:v\in V(G)\}$ and $L=G\circ\mathcal{H}$ is connected,
$\left\vert V(L)\right\vert \geq3$, has no $4$-cycles, and belongs to
$\mathbf{W}_{2}$, then, by Proposition \ref{prop3}, every $H_{v}$ should be
isomorphic to $K_{2}$, i.e., $L=G\circ K_{2}$. Actually, it has been
strengthened as follows.

\begin{theorem}
\cite{Hartnell} Let $L$ be a connected graph without $4$-cycles. The graph $L$
is in class $\mathbf{W}_{2}$ if and only if $L$ is isomorphic to $K_{2}$,
$C_{5}$ or $L=G\circ K_{2}$, for some graph $G$.
\end{theorem}

\begin{corollary}
\label{cor2}If $G$ has non-empty edge set, then $G\circ K_{p}$ is
$1$-well-covered if and only if $p\geq2$.
\end{corollary}

If $G_{1},G_{2},...,G_{p}$ are pairwise vertex disjoint graphs, then their
\textit{join }(or\textit{ Zykov sum}) is the graph $G=G_{1}+G_{2}+\cdots
+G_{p}$ with $V(G)=V(G_{1})\cup V(G_{2})\cup\cdots\cup V(G_{p})$ and
$E(G)=E(G_{1})\cup E(G_{2})\cup\cdots\cup E(G_{p})\cup\{v_{i}v_{j}:v_{i}\in
V(G_{i}),v_{j}\in V(G_{j}),1\leq i<j\leq p\}$.

\begin{proposition}
\cite{ToppLutz}\label{th5} The graph $G_{1}+G_{2}+\cdots+G_{p}$ is
well-covered if and only if each $G_{k}$ is well-covered and $\alpha\left(
G_{i}\right)  =\alpha\left(  G_{j}\right)  $ for every $i,j\in\left\{
1,2,...,p\right\}  $.
\end{proposition}

\begin{proposition}
The graph $G_{1}+G_{2}+\cdots+G_{p}$ belongs to $\mathbf{W}_{2}$ if and only
if each $G_{k}\in\mathbf{W}_{2}$ and $\alpha\left(  G_{i}\right)
=\alpha\left(  G_{j}\right)  $ for every $i,j\in\left\{  1,2,...,p\right\}  $.
\end{proposition}

\begin{proof}
Clearly, if each $G_{k}$ is a complete graph, then $G=G_{1}+G_{2}+\cdots
+G_{p}\in\mathbf{W}_{2}$.

Assume that at least one of $G_{k}$ is not a complete graph. By Proposition
\ref{th5}, we infer that, necessarily, every $G_{k}$ must be well-covered, and
$2\leq\alpha\left(  G_{i}\right)  =\alpha\left(  G_{j}\right)  $ for every
$1\leq i<j\leq p$. Consequently, taking into account the definition of the
Zykov sum, we get $\Omega\left(  G\right)  =\Omega\left(  G_{1}\right)
\cup\Omega\left(  G_{2}\right)  \cup\cdots\cup\ \Omega\left(  G_{p}\right)  $.

Suppose that $G\in\mathbf{W}_{2}$, and let $A$ be a non-maximum independent
set $A$ in some $G_{k}$ and $v\in V\left(  G_{k}\right)  -A$. By Theorem
\ref{th1}\emph{(vii)}, there exists some $S\in\Omega(G)$ such that $A\subset
S$ and $v\notin S$. Since each vertex of $A$ is joined by an edge to every
vertex of $G_{i}$, $i\neq k$, we get that $S\in\Omega(G_{k})$. Therefore,
every $G_{k}$ must be in $W_{2}$, according to Theorem \ref{th1}\emph{(vii)}.

The converse can be obtain in a similar way.
\end{proof}

\begin{corollary}
\cite{Staples} If $G_{1},G_{2}\in\mathbf{W}_{2}$ are such that $\alpha\left(
G_{1}\right)  =\alpha\left(  G_{2}\right)  $, then $G_{1}+G_{2}$ belongs to
$\mathbf{W}_{2}$.
\end{corollary}

Let $G(H,v)$ denote the graph obtained by identifying each vertex of $G$ with
the vertex $v$ of a copy of $H$. $G(H,v)$ it is the $G$-\textit{concatenation}
\textit{of the graph} $H$ \textit{on the vertex} $v$ \cite{Yi}. Clearly,
$G(H,v)$ is connected if and only if both $G$ and $H$ are connected.

\begin{lemma}
\label{lem3}Let $G$ be a connected graph of order $n\geq2$, $\left\vert
V\left(  H\right)  \right\vert \geq2$, and $v\in V\left(  H\right)  $.

\emph{(i)} If $v$ is not in all maximum independent sets of $H$, then
$\alpha\left(  G(H,v)\right)  =n\cdot\alpha\left(  H\right)  $;

\emph{(ii)} If $v$ belongs to every maximum independent set of $H$, then
\[
\alpha\left(  G(H,v)\right)  =n\cdot\left(  \alpha\left(  H\right)  -1\right)
+\alpha\left(  G\right)  .
\]

\end{lemma}

\begin{proof}
\emph{(i)} Let $A$ be a maximum independent set in $H$ with $v\notin A$, and
$S$ be a maximum independent set in $G(H,v)$. First, $n\cdot\alpha\left(
H\right)  =n\cdot\left\vert A\right\vert \leq\alpha\left(  G(H,v)\right)  $,
because the union of $n$ times $A$ is independent in $G(H,v)$.

Since $S$ is of maximum size, it follows that, for every copy of $H$, $S\cap
V\left(  H\right)  $ is non-empty and independent. Consequently, we obtain%
\[
n\cdot\alpha\left(  H\right)  \leq\alpha\left(  G(H,v)\right)  \leq n\cdot
\max\left\vert S\cap V\left(  H\right)  \right\vert \leq n\cdot\alpha\left(
H\right)  ,
\]
as claimed.

\emph{(ii)} Let $A$\ be a maximum independent set in $G(H,v)$. Then $V\left(
G\right)  \cap A$ is independent in $G$ and
\[
\left\vert A\right\vert =\left\vert V\left(  G\right)  \cap A\right\vert
\cdot\alpha\left(  H\right)  +\left(  n-\left\vert V\left(  G\right)  \cap
A\right\vert \right)  \cdot\left(  \alpha\left(  H\right)  -1\right)
=n\cdot\left(  \alpha\left(  H\right)  -1\right)  +\left\vert V\left(
G\right)  \cap A\right\vert .
\]

On the other hand, one can enlarge a maximum independent set $S$ of $G$ to an
independent set $U$ in $G(H,v)$, whose cardinality is
\[
\left\vert U\right\vert =\left\vert S\right\vert \cdot\alpha\left(  H\right)
+\left(  n-\left\vert S\right\vert \right)  \cdot\left(  \alpha\left(
H\right)  -1\right)  =n\cdot\left(  \alpha\left(  H\right)  -1\right)
+\left\vert S\right\vert =n\cdot\left(  \alpha\left(  H\right)  -1\right)
+\alpha\left(  G\right)  .
\]
Since $\left\vert V\left(  G\right)  \cap A\right\vert \leq$ $\alpha\left(
G\right)  $, we conclude with $\alpha\left(  G(H,v)\right)  =n\cdot\left(
\alpha\left(  H\right)  -1\right)  +\alpha\left(  G\right)  $.
\end{proof}

By definition, if $G$ is well-covered and $uv\in E\left(  G\right)  $, then
$u$ and $v$ belong to different maximum independent sets. Therefore, only
isolated vertices, if any, are contained in all maximum independent sets of a
well-covered graph. Thus Lemma \ref{lem3}\emph{(i) }concludes the following.

\begin{corollary}
\label{cor1}If $G$ is a connected graph of order $n\geq2$, and $H\neq K_{1}$
is well-covered, then $\alpha\left(  G(H,v)\right)  =n\cdot\alpha\left(
H\right)  $.
\end{corollary}

The concatenation of two well-covered graphs is not necessarily well-covered.
For example, $K_{2}$ and $C_{4}$ are well-covered, while the graph
$K_{2}\left(  C_{4};v\right)  $ is not well-covered, because $\left\{
v_{1},v_{2},v_{3}\right\}  $ is a maximal independent set of size less than
$\alpha\left(  K_{2}\left(  C_{4};v\right)  \right)  =4$ (see Figure
\ref{fig2}). \begin{figure}[h]
\setlength{\unitlength}{1cm}\begin{picture}(5,1.1)\thicklines
\multiput(2,0)(1,0){4}{\circle*{0.29}}
\multiput(2,1)(1,0){4}{\circle*{0.29}}
\put(2,0){\line(1,0){3}}
\put(2,0){\line(0,1){1}}
\put(2,1){\line(1,0){1}}
\put(3,0){\line(0,1){1}}
\put(4,0){\line(0,1){1}}
\put(4,1){\line(1,0){1}}
\put(5,0){\line(0,1){1}}
\put(1.7,1){\makebox(0,0){$v_{1}$}}
\put(3.7,0.35){\makebox(0,0){$v_{2}$}}
\put(5.35,1){\makebox(0,0){$v_{3}$}}
\put(1,0.5){\makebox(0,0){$G_{1}$}}
\multiput(7,0)(1,0){6}{\circle*{0.29}}
\multiput(7,1)(1,0){2}{\circle*{0.29}}
\multiput(11,1)(1,0){2}{\circle*{0.29}}
\put(7,0){\line(1,0){5}}
\put(7,0){\line(0,1){1}}
\put(7,1){\line(1,0){1}}
\put(8,1){\line(1,-1){1}}
\put(10,0){\line(1,1){1}}
\put(11,1){\line(1,0){1}}
\put(12,0){\line(0,1){1}}
\put(9.1,0.25){\makebox(0,0){$a_{1}$}}
\put(12.35,1){\makebox(0,0){$a_{2}$}}
\put(11,0.35){\makebox(0,0){$x$}}
\put(6.3,0.5){\makebox(0,0){$G_{2}$}}
\end{picture}\caption{$G_{1}=K_{2}\left(  C_{4};v\right)  $ and $G_{2}%
=K_{2}\left(  C_{5};v\right)  $.}%
\label{fig2}%
\end{figure}
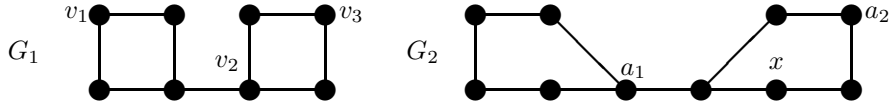

Similarly, the concatenation of two graphs from $\mathbf{W}_{2}$ is not
necessarily in $\mathbf{W}_{2}$. For instance, $K_{2},C_{5}\in$ $\mathbf{W}%
_{2}$, but there is no maximum independent set $S$ in $K_{2}\left(
C_{5};v\right)  $ such that $\left\{  a_{1},a_{2}\right\}  \subset S$ and
$x\notin S$, and hence, by Theorem \ref{th1}\emph{(vii)}, the graph
$K_{2}\left(  C_{5};v\right)  $ is not in $\mathbf{W}_{2}$ (see Figure
\ref{fig2}). However, $K_{2}\left(  C_{5};v\right)  $ is in $\mathbf{W}_{1}$,
i.e., it is well-covered.

\begin{theorem}
\label{th7}\emph{(i)} If $H\in\mathbf{W}_{2}$, then the graph $G(H,v)$ belongs
to $\mathbf{W}_{1}$.

\emph{(ii)} If $H\in\mathbf{W}_{3}$, then the graph $G(H,v)$ belongs to
$\mathbf{W}_{2}$.
\end{theorem}

\begin{proof}
If $H$ is a complete graph, then both \emph{(i)} and \emph{(ii)} are true,
according to Propositions \ref{prop3} and \ref{prop2}, respectively, because
$G(K_{p},v)=G\circ K_{p}$.

Assume that $H$ is not complete, and let $V\left(  G\right)  =\left\{
v_{i}:i=1,2,...,n\right\}  $. By Corollary \ref{cor1}, we have $\alpha\left(
G(H,v)\right)  =n\cdot\alpha\left(  H\right)  $.

\emph{(i)} \ Let $A$ be a non-maximum independent set in $G(H,v)$. We have to
show that $A$ is included in some maximum independent set of $G(H,v)$.

Let $S=S_{1}\cup S_{2}\cup\cdots\cup S_{n}$, where $S_{i}$ is defined as follows:

\begin{itemize}
\item $S_{i}$ is a maximum independent set in the copy $H_{v_{i}}$ of $H$;

\item $v_{i}\notin S_{i}$, whenever $A\cap V\left(  H_{v_{i}}\right)
=\emptyset$; $S_{i}$ exists, since $H$ is well-covered;

\item if $v_{i}\in A\cap V\left(  H_{v_{i}}\right)  $, then $A\cap V\left(
H_{v_{i}}\right)  \subseteq S_{i}$; such $S_{i}$ exists, because $H$ is well-covered;

\item if $v_{i}\notin A\cap V\left(  H_{v_{i}}\right)  \neq\emptyset$, then
$A\cap V\left(  H_{v_{i}}\right)  \subseteq S_{i}$ and $v_{i}\notin S_{i}$; in
accordance with Theorem \ref{th1}\emph{(vii)}, such $S_{i}$ exists, because
$H$ is in $\mathbf{W}_{2}$.
\end{itemize}

Consequently, $S$ is a maximum independent set in $G(H,v)$, because all
$S_{i}$ are independent and pairwise disjoint, each one of size $\alpha\left(
H\right)  $, and $A\subset S$. Therefore, $G(H,v)$ is well-covered.

\emph{(ii)} \ Let $A$ be a non-maximum independent set in $G(H,v)$ and
$x\notin A$. We show that $A$ is included in some maximum independent set of
$G(H,v)$ that does not contain the vertex $x$, and thus, by Theorem
\ref{th1}\emph{(vii)}, we obtain that $G(H,v)$ belongs to $\mathbf{W}_{2}$.

Let $S=S_{1}\cup S_{2}\cup\cdots\cup S_{n}$, where $S_{i}$ is defined as follows:

\begin{itemize}
\item $S_{i}$ is a maximum independent set in the copy $H_{v_{i}}$ of $H$;

\item if $A\cap V\left(  H_{v_{i}}\right)  =\emptyset$ and $x\notin V\left(
H_{v_{i}}\right)  $, then $v_{i}\notin S_{i}$; $S_{i}$ exists, because $H$ is well-covered;

\item if $v_{i}\notin A\cap V\left(  H_{v_{i}}\right)  \neq\emptyset$ and
$x\notin V\left(  H_{v_{i}}\right)  $, then $A\cap V\left(  H_{v_{i}}\right)
\subseteq S_{i}$ and $v_{i}\notin S_{i}$; $S_{i}$ exists, since $H$ is in
$\mathbf{W}_{2}$;

\item if $x=v_{i}$, then $A\cap V\left(  H_{v_{i}}\right)  \subseteq S_{i}$
and $v_{i}\notin S_{i}$; $S_{i}$ exists, because $H$ is in $\mathbf{W}_{2}$;

\item if $x\in V\left(  H_{v_{i}}\right)  -\left\{  v_{i}\right\}  $, then
$A\cap V\left(  H_{v_{i}}\right)  \subseteq S_{i}$ and $x,v_{i}\notin S_{i}$;
$S_{i}$ exists, since $A\cap V\left(  H_{v_{i}}\right)  ,\left\{  x\right\}  $
and $\left\{  v_{i}\right\}  $ are independent and disjoint, and $H$ belongs
to $\mathbf{W}_{3}$.
\end{itemize}

Consequently, $S$ is a maximum independent set in $G(H,v)$ (because all
$S_{i}$ are independent and pairwise disjoint, each one of size $\alpha\left(
H\right)  $), $x\notin S$ and $A\subset S$. Therefore, $G(H,v)$ is in
$\mathbf{W}_{2}$.
\end{proof}

\section{Conclusions}

We proved that a well-covered $G$ without isolated vertices satisfies
$Shed\left(  G\right)  =V\left(  G\right)  $ if and only if $G$ is
$1$-well-covered. On the other hand, there exist well-covered graphs without
shedding vertices; e.g., $C_{4}$ and $C_{7}$. This leads to the following.

\begin{problem}
Find all well-covered graphs having no shedding vertices.
\end{problem}

By definition, every graph from class $\mathbf{W}_{2}$ has two disjoint
maximum independent sets at least, while some have even three pairwise
disjoint maximum independent sets (e.g., $P_{n}\circ K_{2}$, for $n\geq1$).
However, $C_{5}$ is in $\mathbf{W}_{\mathbf{2}}$, but has no enough vertices
for three maximum independent sets pairwise disjoint.

\begin{theorem}
\label{th8}The corona $G=H\circ K_{1}$ has two disjoint maximum independent
sets if and only if $H$ is a bipartite graph.
\end{theorem}

\begin{proof}
In what follows, for each $A\in V\left(  H\right)  $, the set $N\left(
A\right)  -V\left(  H\right)  $ is denoted by\ $A^{\ast}$. It is not difficult
to see that every independent set in $G$ is of the form $X\cup Y^{\ast}$,
where $X\cap Y=\emptyset$ and $X$ is independent in $H$. Moreover, $X\cup
Y=V\left(  H\right)  $ if and only if $X\cup Y^{\ast}$ is a maximum
independent set in $G$.

\textquotedblleft\textit{If}\textquotedblright\ Let $\left\{  A,B\right\}  $
be a bipartition of $H$. Then $A\cup B^{\ast}$ and $A^{\ast}\cup B$ are two
disjoint maximum independent sets of $G$.

\textquotedblleft\textit{Only if}\textquotedblright\ Let $X_{1}\cup
Y_{1}^{\ast}$ and $X_{2}\cup Y_{2}^{\ast}$ be two disjoint maximum independent
sets of $G$. Suppose that there exists some vertex $v\in V\left(  H\right)
-\left(  X_{1}\cup X_{2}\right)  $. Hence $v\in Y_{1}\cap Y_{2}$, in
contradiction with the fact that $X_{1}\cup Y_{1}^{\ast}$ and $X_{2}\cup
Y_{2}^{\ast}$ are disjoint. Therefore, we get $V\left(  H\right)  =X_{1}\cup
X_{2}$. Since $X_{1}\cap X_{2}=\emptyset$, the pair $\left\{  X_{1}%
,X_{2}\right\}  $ is a bipartition of $H$.
\end{proof}

Clearly, $C_{7}$ has two disjoint maximum independent sets. Thus Theorems
\ref{th4},\ref{th8} provide us with a complete description of well-covered
graphs of girth $\geq6$ containing a pair of disjoint maximum independent sets.

\begin{problem}
\label{prob1}Characterize well-covered graphs of girth $\leq5$ with two
disjoint maximum independent sets at least.
\end{problem}

If $G$ is disconnected, then the only $\mathbf{W}_{2}$ graphs with
$\alpha(G)=2$ are $G=K_{n}\cup K_{m}$, where $m,n\geq2$.

The graph $G$ is \textit{locally triangle-free} if $G-N\left[  v\right]  $ is
triangle-free for any vertex $v\in V\left(  G\right)  $ \cite{Hoang2016b}. If
$\alpha(G)\leq2$, the structure of locally triangle-free graphs belonging to
$\mathbf{W}_{2}$ is as follows.

\begin{proposition}
\label{prop5}\cite{Hoang2016b} Let $G$ be a locally triangle-free graph in
$\mathbf{W}_{2}$ of order $n$. Then,

\emph{(i)} If $\alpha(G)=1$, then $G$ is $K_{n}$ with $n\geq2$;

\emph{(ii)} If $\alpha(G)=2$, then $G$ is the complement of $C_{n}$ with
$n\geq4$.
\end{proposition}

On the other hand, there exist graphs in $\mathbf{W}_{2}$, which are different
from complement of cycles, for instance, $P_{2}\circ K_{2}$. It motivates the following.

\begin{problem}
Characterize connected $\mathbf{W}_{2}$ graphs with $\alpha=2$.
\end{problem}

Notice that every $G\in\left\{  C_{3},C_{5},P_{2}\circ K_{2}\right\}  $
belongs to $\mathbf{W}_{2}$ and satisfy $\alpha(G)+\mu(G)=\left\vert
V(G)\right\vert -1$. Clearly, if such $G$\ is disconnected, then all its
components but one are $K_{2}$.

\begin{problem}
Find all connected graphs $G\in\mathbf{W}_{2}$ satisfying $\alpha
(G)+\mu(G)=\left\vert V(G)\right\vert -1$.
\end{problem}

It seems promising to extend our findings in the framework of $\mathbf{W}_{k}$
classes for $k\geq3$. \ For instance, the same way we proved Theorem
\ref{prop11}\emph{(vii)} one can show the following.

\begin{theorem}
Let $G\in\mathbf{W}_{k}$. \emph{If }$A\subseteq B$, then%
\[
\left\vert N(A)\right\vert -\left(  k-1\right)  \left\vert A\right\vert
\leq\left\vert N(B)\right\vert -\left(  k-1\right)  \left\vert B\right\vert
\]
for every independent set $B\subseteq V\left(  G\right)  $.
\end{theorem}

Taking into account Theorem \ref{th7}, we propose the following.

\begin{conjecture}
If $H\in\mathbf{W}_{k}$, then the concatenation $G(H,v)$ belongs to
$\mathbf{W}_{k-1}$.
\end{conjecture}

\section{Acknowledgements}

We express our gratitude to Zakir Deniz, who has pointed out that the
well-coverdness condition in Theorem \ref{th12}\emph{(iii) }is necessary. Our
special thanks are to Tran Nam Trung for: an example that has helped us to
formulate Problem \ref{prob1} more precisely; introducing us to Proposition
\ref{prop5}; and some valuable suggestions that have brought us to the proof
of Theorem \ref{th12}\emph{(iv),(v)}.


\begin{thebibliography}{99}                                                                                               %


\bibitem {Baker}J. Baker, K. N. Vander Meulen, A. Van Tuyl, \emph{Shedding
vertices of vertex decomposable graphs}, arXiv:1606.04447 [math.CO] (2016) 23 pp.

\bibitem {Berge 1978}C. Berge, \emph{Regularizable graphs I}, Discrete
Mathematics \textbf{23} (1978) 85--89.

\bibitem {Berge1982}C. Berge, \emph{Some common properties for regularizable
graphs, edge-critical graphs and B-graphs}, Annals of Discrete Mathematics
\textbf{12} (1982) 31--44.

\bibitem {BermudoFernau2012}S. Bermudo, H. Fernau, \emph{Lower bounds on the
differential of a graph}, Discrete Mathematics \textbf{312} (2012) 3236--3250.

\bibitem {Bermudo2014}S. Bermudo, J. C. Hern\'{a}ndez-G\'{o}mez, J. M.
Rodr\'{\i}guez, J. M. Sigarreta, \emph{Relations between the differential and
parameters in graphs}, Electronic Notes in Discrete Mathematics \textbf{46}
(2014) 281--288.

\bibitem {Biyi}T. Biyiko\u{g}lu, Y. Civan, \emph{Vertex-decomposable graphs,
codismantlability, Cohen-Macaulayness, and Castelnouvo-Mumford regularity},
The Electronic Journal of Combinatorics \textbf{21} (2014) \#P1.1.

\bibitem {CaCrRey2016}I. D. Castrill\'{o}n, R. Cruz, E. Reyes, \emph{On
well-covered, vertex decomposable and Cohen-Macaulay graphs}, The Electronic
Journal of Combinatorics \textbf{23} (2016) \#P2.39.

\bibitem {DE2016}Z. Deniz, T. Ekim, \emph{Stable equimatchable graphs},
arXiv:1602.09127\ [cs.DM] (2016) 17 pp.

\bibitem {Do}A. Dochtermann, A. Engstr\"{o}m, \emph{Algebraic properties of
edge ideals via combinatorial topology}, The Electronic Journal of
Combinatorics \textbf{16} (2009) \#R2.

\bibitem {EVMVT2015}J. Earl, K. N. Vander Meulen, A. Van Tuyl,
\emph{Independence complexes of well-covered circulant graphs}, Experimental
Mathematics \textbf{25} (2016) 441--451.

\bibitem {Estrada}M. Estrada, R. H. Villarreal, \emph{Cohen-Macaulay bipartite
graphs}, Arch. Math. \textbf{68} (1997) 124--128.

\bibitem {Favaron1982}O. Favaron, \emph{Very well-covered graphs}, Discrete
Mathematics \textbf{42} (1982) 177--187.

\bibitem {FinHarNow}A. Finbow, B. Hartnell, R. J. Nowakowski, \emph{A
characterization of well-covered graphs of girth }$5$\emph{\ or greater},
Journal of Combinatorial Theory B \textbf{57} (1993) 44--68.

\bibitem {FruchtHarary}R. Frucht, F. Harary, \emph{On the corona of two
graphs}, Aequationes Mathematicae \textbf{4} (1970) 322--324.

\bibitem {Hartnell}B. L. Hartnell, \emph{A characterization of the }%
$1$\emph{-well-covered graphs with no }$4$\emph{-cycles}, Graph Theory Trends
in Mathematics, 219--224, 2006 Birkh\"{a}user-Verlag Basel/Switzerland.

\bibitem {Hoang2015}D. T. Hoang, N. C. Minh, T. N. Trung, \emph{Cohen-Macaulay
graphs with large girth}, Journal of Algebra and its Applications \textbf{14}
(2015) 1550112, 16 pp.

\bibitem {Hoang2016a}D. T. Hoang, T. N. Trung, \emph{A characterization of
triangle-free Gorenstein graphs and Cohen--Macaulayness of second powers of
edge ideals}, Journal of Algebraic Combinatorics \textbf{43} (2016) 325--338.

\bibitem {Hoang2016b}D. T. Hoang, T. N. Trung, \emph{Buchsbaumness of the
second powers of edge ideals}, arXiv:1606.02815v2 [math.CO] (2016) 20 pp.

\bibitem {LevMan2007}V. E. Levit, E. Mandrescu, \emph{Some structural
properties of very well-covered graphs}, Congressus Numerantium \textbf{186}%
\ (2007) 97--106.

\bibitem {Mashburn2006}J. L. Mashburn, T. W. Haynes, S. M. Hedetniemi, S. T.
Hedetniemi, P. J. Slater, \emph{Differentials in graphs}, Utilitas Mathematica
\textbf{69} (2006) 43--54.

\bibitem {Pinter1991}M. R. Pinter, $W_{2}$ \emph{graphs and strongly
well-covered graphs: two well-covered graph subclasses}, Vanderbilt Univ.
Dept. of Math. Ph.D. Thesis, 1991.

\bibitem {Pinter1992}M. Pinter, \emph{Planar regular one-well-covered graphs},
Congressus Numerantium \textbf{91} (1992) 159--187.

\bibitem {Pinter1995}M. Pinter, \emph{A class of planar well-covered graphs
with girth four}, Journal of Graph Theory \textbf{19} (1995) 69--81.

\bibitem {plum}M. D. Plummer, \emph{Some covering concepts in graphs}, Journal
of Combinatorial Theory \textbf{8} (1970) 91--98.

\bibitem {PrToppVest1996}E. Prisner, J. Topp, P. D. Vestergaard,
\emph{Well-covered simplicial, chordal, and circular arc graphs}, Journal of
Graph Theory 21 (1996) 113--119.

\bibitem {StaplesThesis}J. W. Staples, \emph{On some subclasses of
well-covered graphs}, Ph.D. Thesis, 1975, Vanderbilt University.

\bibitem {Staples}J. W. Staples, \emph{On some subclasses of well-covered
graphs}, Journal of Graph Theory \textbf{3} (1979) 197--204.

\bibitem {ToppLutz}J. Topp, L. Volkman, \emph{On the well-coveredness of
products of graphs}, Ars Combinatoria \textbf{33} (1992) 199--215.

\bibitem {Vi}R. Villarreal, \emph{Monomial Algebras, Monographs and Textbooks
in Pure and Applied Mathematics}, Vol. \textbf{238}, Marcel Dekker, New York (2001).

\bibitem {Woodroofe2009}R. Woodroofe, \emph{Vertex decomposable graphs and
obstructions to shellability}, Proceedings of the American Mathematical
Society \textbf{137} (2009) 3235--3246.

\bibitem {Woodroofe2011}R. Woodroofe, \emph{Chordal and sequentially
Cohen-Macaulay clutters}, The Electronic Journal of Combinatorics \textbf{18}
(2011), \#R208.

\bibitem {Yi}Yi Wang, Bao-Xuan Zhu, \emph{On the unimodality of independence
polynomials of some graphs}, European Journal of Combinatorics \textbf{32}
(2011) 10--20.
\end{thebibliography}
\end{document}